\documentclass[11pt]{amsart}
\usepackage{amsmath,amssymb, amsthm, mathabx,amscd, amsfonts, amssymb, hyperref, mathrsfs, textcomp, epsfig, graphicx, IEEEtrantools, tikz, verbatim, subcaption, pinlabel}
\newtheorem*{thm}{Theorem}

\newtheorem{prop}{Proposition}[section]
\newtheorem{lemma}{Lemma}[section]

\newtheorem{cor}{Corollary}[section]

\theoremstyle{definition}

\theoremstyle{remark}
\newtheorem{remark}{Remark}[section]
\newtheorem{example}{Example}[section]
\newtheorem{exercise}{Exercise}[section]

\newcommand{\matZ}{\mathbb{Z}}

\newcommand{\lk}{ {\rm lk}}

\newcommand{\matR}{\mathbb{R}}
\newcommand{\matC}{\mathbb{C}}
\newcommand{\matS}{\mathbb{S}}

\newcommand{\matH}{\mathbb{H}}
\newcommand{\matRP}{\mathbb{RP}}

\begin{document}

\title[Symplectic hyperbolic $\pi$-orbifolds]{Symplectic four-manifolds \\ and hyperbolic $\pi$-orbifolds}

\author{Francesco Lin}
\address{Department of Mathematics, Columbia University} 
\email{flin@math.columbia.edu}

\author{Bruno Martelli}
\address{Dipartimento di Matematica, Universit\`a di Pisa} 
\email{bruno.martelli@unipi.it}
\begin{abstract}
We construct closed symplectic 4-manifolds admitting the structure of hyperbolic $\pi$-orbifold, addressing LeBrun's conjecture. Our simplest example has Euler characteristic $304$, is tessellated by $64$ copies of the right-angled hyperbolic $120$-cell and admits a four-fold branched covering over a product of two surfaces of genus $4$. 
\end{abstract}

\maketitle

LeBrun conjectured in \cite{LeBrun} that all closed oriented hyperbolic four-manifolds have vanishing Seiberg-Witten invariants \cite{Wit, Mor}. While this has been verified in some specific examples \cite{AgolLin,BFS, LM}, the evidence towards it is still very scarce, and it is natural to wonder whether the conjecture actually holds true or not. In this sense, the foundational result of Taubes \cite{Tau} suggests the very natural goal of searching for a hyperbolic four-manifold admitting a symplectic structure, i.e. a closed non-degenerate $2$-form $\omega$. This is a particularly challenging endeavor because these two geometric structures (hyperbolic and symplectic) are very different in nature, and the techniques commonly used to construct them seem to have very little interactions.
\par
In this direction, Stover \cite{Sto} recently showed that the complex projective plane $\mathbb{CP}^2$ (which admits the natural Fubini-Study symplectic form) can be realized as a hyperbolic \textit{orbifold} with rather complicated isotropy groups, including the dihedral group with eight elements inside $SO(3)\subset SO(4)$. Specifically, he explicitly describes a smooth four-manifold triangulated by $60$ copies of a fixed hyperbolic Coxeter simplex, which is then verified to be diffeomorphic to $\mathbb{CP}^2$ using the computational software Regina \cite{Reg}.
\par
The goal of this paper is to show that in fact there are explicit (much larger) examples of symplectic four-manifolds which can be realized as a hyperbolic orbifold with the simplest possible kind of singularity, namely with only isotropy group $\matZ/2\matZ$ generated by an isometric involution acting at the fixed points with eigenvalues $\{-1,-1,1,1\}$. We will refer to this as a hyperbolic $\pi$-\textit{orbifold structure}. Our main theorem is then the following.
\begin{thm}
There exist infinitely many closed symplectic four-manifolds admitting the structure of hyperbolic $\pi$-orbifold.
\end{thm}
These manifolds are obtained as covering spaces of a specific `simplest' example $M$ having Euler characteristic $\chi(M)=304$ and for which the singular locus of the $\pi$-orbifold structure is the disjoint union of $160$ surfaces of genus $2$ and $40$ surfaces of genus $3$. The theorem will be proven `with bare hands' by combining constructive techniques in hyperbolic geometry and symplectic topology.  
\\
\par
More specifically, on the hyperbolic geometry side, we will provide a concrete description of $M$ as the union of $2^6=64$ copies the right-angled $120$-cell (a remarkable polyhedron with $120$ dodecahedral facets) via the construction of manifolds via colourings of polyhedra introduced in \cite{V}. 
\par
Combining this description with beautiful combinatorial properties of the $120$-cell (closely related to its structure as discrete Hopf fibration \cite{SchSeg}) we will be able to prove that $M$ admits a four-fold branched covering
\begin{equation*}
p\colon M\rightarrow \Sigma\times\Sigma
\end{equation*}
where $\Sigma$ is a closed surface of genus $4$ and $p$ has very simple singularities, namely (using coordinates $(z,w)$ in $\mathbb{C}^2)$ of the form\footnote{In particular, throughout the paper we will consider `branched coverings' which are not necessarily cyclic (the ones most commonly considered in the four-dimensional topology literature).}
\begin{equation*}
(z,w)\mapsto(z^2,w)\text{ or }(z^2,w^2).
\end{equation*}
Furthermore, the 
ramification locus in $\Sigma\times\Sigma$ is a large but very concrete collection of transversely intersecting product tori.
\par
On the symplectic topology side, with this explicit description at hand we will be able to construct a symplectic form $\omega$ on $M$ starting from the natural product one on $\Sigma\times\Sigma$ following techniques introduced by (among the others) Gompf \cite{Gom} and Auroux \cite{Aur}.
\\
\par
Regarding LeBrun's conjecture, by construction each manifold $X$ in our Theorem admits a branched double cover $\tilde{X}$ which is a genuine hyperbolic manifold (for example, for the simplest example $M$ we have $\chi(\tilde{M})=1088$). Unfortunately our techniques run into a serious obstruction when trying to construct a symplectic structure on $\tilde{X}$. The ramification locus of $\tilde{p}:\tilde{X}\rightarrow X$ is the singular set of the $\pi$-orbifold structure of $X$, a collection of disjoint totally geodesics surfaces $\{\Sigma_i\}$ (therefore having self-intersection zero and genus at least $2$) on which the first Chern class of the symplectic structure $c_1(X,\omega)$ (which is defined as $c_1(TX,J)$ for an $\omega$-compatible almost complex structure $J$ on $X$) evaluates to zero. This implies that the adjunction identity
\begin{equation*}
-\langle c_1(X,\omega),[\Sigma_i]\rangle+[\Sigma_i]^2+2-2g(\Sigma_i)=0
\end{equation*}
does \textit{not} hold for the surface $\Sigma_i$, which is therefore not isotopic to a symplectic surface (see for example \cite[Theorem 2.51]{Wen}). Even worse, the same computation (readily generalizing the case of connected branched locus discussed in \cite{RW}) implies that for all the spin$^c$ structures on $\tilde X$ which agree with the pull-back of the canonical one $\mathfrak{s}_{\omega}$ on $X$ on the regular locus of $p$ have negative formal dimension, and therefore have vanishing Seiberg-Witten invariants by definition.
\begin{remark}
Of course, this does not rule out the possibility that some $\tilde{X}$ is indeed symplectic, or more generally has a non-vanishing Seiberg-Witten invariant for some other spin$^c$ structure; we will leave it as a interesting open question to determine the latter in full generality (including higher dimensional moduli spaces). For some examples of infinite families of hyperbolic four-manifolds tasselated by right-angled $120$-cells with vanishing Seiberg-Witten invariants, see \cite{BFS}.
\end{remark}
From a geometric perspective, the essential obstruction seems to stem from fact that the composite branched cover
\begin{equation*}
p\circ\tilde{p}:\tilde{M}\rightarrow \Sigma\times \Sigma
\end{equation*}
has new `bad' singularities with isotropy group given by the Klein-four group in $SO(3)\subset SO(4)$ consisting of diagonal matrices, which acts on the preimage of the ramification locus in an \textit{orientation-reversing} fashion. In this sense, one would expect that trying to apply an analogous strategy to find covers of Stover's examples which are symplectic manifolds would run in the same obstruction (even if one could show that the orbifold locus interacts nicely with the symplectic structure of $\mathbb{C}P^2$).

\vspace{0.3cm}

\textit{Organization of the paper.} Section \ref{colours} provides a general introduction to the construction of manifolds via colourings of polyhedra. In Section \ref{mainexample} we discuss the specific examples of interest, and in Section \ref{symplecticstuff} we prove our main result regarding symplectic structures.

\begin{remark} While the content of Sections \ref{colours} and \ref{symplecticstuff} should be well-known to experts in the respective areas, we designed our exposition in order to make the whole paper accessible to anyone in either field.
\end{remark}

\textit{Acknowledgements. }The first author was partially supported by the grants NSF DMS-2503714 and Simons Foundation TSM-00013131; the project started while he was visiting Universit\`a di Pisa in Fall 2025, and he thanks them for the warm hospitality.
\\
\par
\textit{Tool use. }No LLMs were used during the writing of this paper. Pictures were made with Inkscape.

\section{Colourings of polyhedra}\label{colours}
L\"obell \cite{Lobell} built the first closed hyperbolic 3-manifold in 1931 using right-angled dodecahedra, via a method, later
generalised by Vesnin \cite{V} and other authors, that consists of \emph{colouring} the facets of a polyhedron with some appropriate elements of a vector space over $\mathbb{Z}/2\mathbb{Z}$. This method is flexible enough to work in many different contexts and to allow various topological operations. We show here that it can be used to build interesting 
branched coverings. 

\subsection{Abstract polyhedra}
We would like to define some abstract simple polyhedra, that are not necessarily realizable as honest compact polyhedra in $\matR^n$. This is usually best done via simplicial complexes.

\subsubsection{Simplicial complexes}
Let $[m] = \{1,\ldots, m\}$.
We let $2^{[m]}$ denote the set of subsets of $[m]$. Recall that an \emph{abstract simplicial complex} is a subset $K \subset 2^{[m]}$ closed under taking subsets. A subset $\sigma\in K$ of cardinality $k+1$ is an abstract $k$-\emph{simplex}, and when $k=0,1$ we also call it a \emph{vertex} and an \emph{edge}.

Every abstract simplicial complex $K$ has a \textit{geometric realisation} as a simplicial complex $|K| \subset \matR^{m}$ that is the union of the simplexes $|\sigma|$ with vertices $e_i, i \in \sigma$ as $\sigma \in K$ varies. We indicate with $|K|$ both the resulting topological space and its structure as a simplicial complex. 

\subsubsection{Abstract polyhedra}
We now introduce abstract polyhedra.
Let $K \subset 2^{[m]}$ be an abstract simplicial complex such that $\{i\} \in K$ for every $i\in [m]$ and $|K|$ is PL homeomorphic to $S^{n-1}$. 

The \emph{dual cell decomposition} of $|K|$ is defined by taking the barycentric subdivision $|K|'$ of $|K|$, and defining for every $k$-simplex $\sigma \in K$ its \emph{dual $(n-k-1)$-cell} $F_\sigma$ as the union of the simplexes in $|K|'$ that intersect $|\sigma|$ in its barycenter. This is indeed a $(n-k-1)$-cell, because it is PL homeomorphic to the cone over the \emph{link} $\lk(\sigma)$ of $\sigma$, that is
$$\lk(\sigma) = \{\tau \in K\ |\ \tau \cap \sigma = \emptyset, \tau \cup \sigma \in K \} \cong S^{n-k-1}.$$

Let $P$ be the cone over $|K|$, equipped with its dual cell decomposition. A $k$-cell $F_\sigma$ of $\partial P = |K|$ is called a \emph{$k$-face} of $P$. When $k=0,1,n-2,n-1$ the face is also called \emph{vertex}, \emph{edge}, \emph{ridge}, and \emph{facet}. 

We call the resulting $P$ an \emph{abstract polyhedron}, and we consider it only up to PL faces-preserving homeomorphisms. For instance, if $K$ is the boundary of a $n$-simplex, $n$-cross-polytope, icosahedron, 600-cell, then $P$ is a $n$-simplex, $n$-cube, dodecahedron, 120-cell respectively. 

We note that $P$ is always \emph{simple}, that is every $k$-face in $P$ is contained in precisely $n-k$ facets. Every simple compact polyhedron $P \subset \matR^n$ can be obtained in this way (up to PL face-preserving homeomorphisms). Simple compact polyhedra $P\subset \matR^n$ are particular kinds of abstract polyhedra; there are however some abstract polyhedra that may not be realised as a compact polyhedron in $\matR^n$.

The $k$-face $F_\sigma$ of an abstract polyhedron $P$ is itself an abstract polyhedron, obtained from the simplicial complex $\lk_\sigma$. We write $F_i = F_{\{i\}}$ for simplicity. The facets of the abstract polyhedron $P$ are $F_1,\ldots, F_m$, dual to the vertices $1,\ldots, m$ of $K$.

\subsection{Colourings}
To fix notation, from now on $K \subset 2^{[m]}$ is always an abstract simplicial complex such that $\{i\} \in K$ for all $i$ and $|K|$ is PL homeomorphic to $S^{n-1}$, and $P$ is its associated abstract polyhedron. 

\subsubsection{Definition}
Let $V$ be a $c$-dimensional vector space over the field $\matZ/2\matZ$. A $c$-\emph{colouring} of $K$ is a map $\lambda \colon [m] \to V$ that sends the vertices of every 
$\sigma \in K$ to independent vectors. This corresponds to colouring the facets $F_i$ of $P$ with vectors in $V$, such that the colours on any set of incident
facets are independent vectors (we say that some faces are incident if their intersection is non-empty).

We will always suppose that the image of $\lambda$ spans $V$, since this is easily achieved by restricting the codomain.

As an example, we can take $V = (\matZ/2\matZ)^m$ and $\lambda(i) = e_i$. As we will see, for topological reasons it is often preferable to keep $c = \dim V$ as low as possible, so this example is in many cases not the preferred one.

\subsubsection{The manifold $N$}
A colouring $\lambda$ of $K$ defines a PL manifold $N$ that decomposes into $2^c$ copies of $P$, as follows. For every $v\in V$, let $P_v$ be an identical copy of $P$. We start with the disjoint union of $P_v$ as $v \in V$. Then, for every $i\in [m]$ and $v\in V$, we glue the facet $F_i$ in $P_v$ with the facet $F_i$ in $P_{v + \lambda(i)}$ via the obvious identity map. Let $N$ be the resulting topological space.

One verifies easily that $N$ is in fact a compact PL manifold without boundary 
tessellated into the $2^c$ abstract polyhedra $P_v$. If $P$ is a right-angled compact polyhedron in $\matH^n, \matR^n, \matS^n$, then $N$ is naturally a compact hyperbolic, flat, spherical manifold.

Of course, we always have $c\geq n$. When $c=n$ the manifold $N$ is called a \emph{small cover}.
The reason for adopting this terminology is that $P$ can be interpreted as an orbifold, the natural map $N \to P$ (obtained by identifying each $P_v$ with $P$) is an orbifold covering, every manifold covering of $P$ has degree at least $2^n$, and when it has degree precisely $2^n$ it is constructed in this way.

Small covers are typically easier to study, and they were examined extensively by Davis and Januszkiewicz \cite{DJ}. 
As we alluded above, if $c$ is very large the manifold $N$ can be very big and often intractable, especially in higher dimension $n\geq4$. 

\subsubsection{Examples}
Two colourings $\lambda, \lambda'$ on $K$ are isomorphic if they are induced by a combination of a linear isomorphism $V \cong V'$ and a simplicial automorphism of $K$. Isomorphic colourings produce the same $N$, hence 
we consider colourings always up to isomorphism. 

We list some examples. We always take $V = (\matZ/2\matZ)^c$ and denote by $e_1, \ldots, e_c$ its canonical basis.
\begin{enumerate}
\item The $n$-simplex has a unique $n$-colouring, with $\lambda(i) =e_i$ for $i=1,\ldots, n$ and $\lambda(n+1) = e_1+ \cdots + e_{n+1}$, and it gives the small cover $N = \matRP^n$, see \cite{DJ}. It also has a unique $(n+1)$-colouring, with $\lambda(i) = e_i$ for all $i$. In this case $N= S^n$. It has no other colourings.
\item The $n$-cube has a $n$-colouring where opposite facets are assigned the same colour $e_i$, and it produces a flat $n$-torus $N$. It also has other small covers: with $n=2$ we can get the Klein bottle \cite{DJ}, and with $n=3$ we can get the half-twist flat 3-manifold \cite{FKS}. With $n=3$ and $c=4$ we can also get the Hantzsche -- Wendt flat 3-manifold \cite{FKS}. If we only use the vectors $e_i$ to colour the $n$-cube, we always get a $n$-torus \cite{IMM}.
\item The $2k$-gon has a unique 2-colouring that produces an orientable surface $N$. The 2-colouring assigns $e_1$ and $e_2$ alternatively to its edges. The resulting surface has genus $k-1$.
\item The dodecahedron has 25 small covers \cite{GS}. Since the dodecahedron may be realised as a compact right-angled polyhedron in $\matH^3$, these are closed hyperbolic 3-manifolds.
\item The 120-cell has 56 orientable small covers \cite{MZ}. Since the 120-cell may be realised as a compact right-angled polyhedron in $\matH^4$, these are closed hyperbolic 4-manifolds.
\end{enumerate}

\subsection{Basic properties}
Let $N$ be a manifold obtained by assigning a colouring $\lambda$ to an abstract polyhedron $P$. Many interesting topological and geometric properties of $N$ can be inferred from $P$ and $\lambda$. We focus on those that are relevant for us here. 

\subsubsection{{Euler characteristic}}
The orbifold Euler characteristic of $P$ is 
$$\chi(P) = \sum_{i=0}^n \frac {f_i}{2^i}$$
where $f_i$ is the number of $i$-faces of $P$. When $P$ has odd dimension we have $\chi(P)=0$. The Euler characteristic of $N$ is then
$$\chi(N) = 2^c \chi(P)$$
where $c=\dim V$.

\subsubsection{Faces of $P$ induce submanifolds of $N$.} \label{submanifolds:subsubsection}
By identifying all the $P_v$ with $P$ we get a natural map $\pi \colon N \to P$, that is in fact an orbifold covering if we consider $P$ with its natural right-angled orbifold structure. 

The preimage $\pi^{-1}(F)$ of a $k$-face $F\subset P$ is a $k$-dimensional submanifold of $N$ tessellated into some copies of $F$.
If $P$ is a right-angled polyhedron in $\matR^n$, $\matH^n$, or $\matS^n$, the submanifold is totally geodesic.

The colouring $\lambda$ of $P$ descends to $F$: every facet of $F$ is of the type $F \cap F_i$ for a unique facet $F_i$ of $P$, and it inherits the colour $\lambda(i)$ of $F_i$. Let $N_F$ be the $k$-manifold produced by this colouring. The preimage $\pi^{-1}(F)$ consists of some copies of $N_F$. The number of copies is equal to the index of the subgroup in $V$ generated by the colours of the facets of $P$ that intersect $F$.

\subsubsection{Covering spaces}
Let $P$ be an abstract polyhedron. Suppose that $\lambda, \lambda'$ are two colourings with vector spaces $V, V'$, such that $\lambda' = f\circ\lambda$ for some surjective homomorphism $f\colon V \to V'$. Let $N, N'$ be the resulting manifolds. We have a natural covering $N \to N'$ with degree the cardinality of $\ker f$, defined by sending $P_v$ to $P_{f(v)}$.

\subsection{Merging facets and branched coverings}\label{branchedcovering}
We now introduce an operation on polyhedra that consists of \emph{merging} some of their facets. This operation combined with colourings will allow us to construct branched coverings between the manifolds. Merging is best described as an operation on simplicial complexes $K$, and then transported to abstract polyhedra $P$.

\subsubsection{Merging facets} \label{merging:subsubsection}
Let $P$ be an abstract polyhedron constructed from some abstract simplicial complex $K \subset 2^{[m]}$ where $|K|$ is PL homeomorphic to $S^{n-1}$. Every surjective map $g\colon [m] \to [m']$ defines a new abstract simplicial complex $K' = \{g(\sigma)\ |\ \sigma \in K\}\subset 2^{[m']}$ {and a surjective simplicial map $g\colon K \to K'$}. Suppose now that $g^{-1}(\sigma')$ is {a} collapsible {subcomplex}, for all $\sigma' \in K'$. By \cite{Cohen}
the space $|K'|$ is also PL homeomorphic to $S^{n-1}$ and hence $K'$ provides a new abstract polyhedron $P'$. 

The new abstract polyhedron $P'$ is obtained from $P$ by \emph{merging} some adjacent facets, according to the surjective map $g$. The facet $F'_i$ of $P'$ is obtained by merging the facets $F_j$ of $P$ with $g(j)=i$. Here are some examples:
\begin{enumerate}
\item By merging $k$ pairs of consecutive edges of a $2k$-gon we get a $k$-gon; 
\item By merging 6 appropriate pairs of adjacent pentagons in a dodecahedron we get a cube, see Figure \ref{merge:fig}; 
\item By merging 4 appropriate triples of pentagons on a dodecahedron we get a tetrahedron. 
\end{enumerate}
We will see below that by merging 20 appropriate 6-tuples of dodecahedra in the 120-cell we get a product of two decagons, and this fact is the core of our construction.

\begin{figure}
 \begin{center}
\centering
\labellist
\small\hair 2pt
\pinlabel $P$ at 90 20
\pinlabel $P'$ at 250 20
\pinlabel $P$ at 500 0
\pinlabel $P'$ at 730 0
\endlabellist
  \includegraphics[width = 12.5 cm]{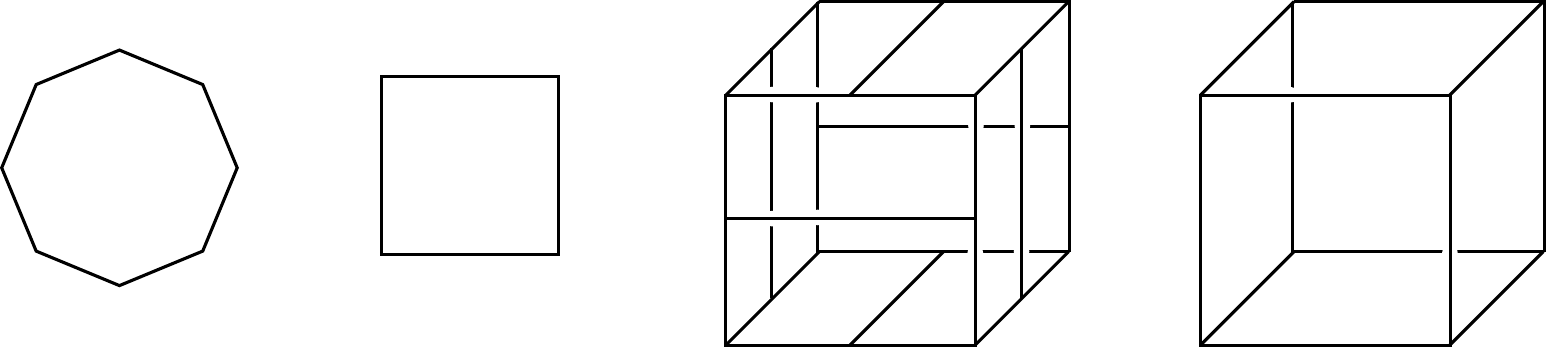}
 \end{center}
 \caption{A square $P'$ obtained by merging pairs of edges of an octagon $P$, and a cube $P'$ obtained by merging pairs of faces of a dodecahedron $P$ (the dodecahedron is described here only combinatorially). In both cases $P$ has a right-angled hyperbolic structure and $P'$ a right-angled Euclidean one.}  \label{merge:fig}
\end{figure}

Let $g_* \colon P \to P'$ be a PL homeomorphism that merges the facets, that is such that $g_*(F_i) \subset F_{g(i)}'$ for every $i$. The \emph{merging locus} in $P'$ is the image of the codimension-2 faces in $P$ separating the pairs of facets that have been merged. For instance, the merging locus in Figure \ref{merge:fig} consists of four middle points in the edges of the square, and 6 middle segments in the faces of the cube.

\subsubsection{Branched coverings} \label{branched:subsubsection}
Let $P'$ be obtained from $P$ by merging some of its facets, according to some map $g\colon [m] \to [m']$.
Let $\lambda, \lambda'$ be colourings of $P, P'$ with vector spaces $V,V'$, producing manifolds $N,N'$, and suppose that there is a homomorphism $f\colon V \to V'$ such that $\lambda'\circ g = f \circ \lambda$. When $P=P'$ and $g={\rm id}$, we have seen above that we get a covering $N \to N'$. 

In this more general setting, we get a branched covering $h\colon N \to N'$, defined by sending $P_v$ to $P_{f(v)}'$ via $g_*$. This map is a branched covering with ramification locus the preimage in $N'$ of the merging locus of $P'$. 

\begin{prop}
   The branched covering is regular (i.e. for every $p,q \in N$ with $h(p)=h(q)$ there is a PL homeomorphism $\varphi \colon N \to N$ with $\varphi(p)=q$ and $h\circ \varphi = h$). 
\end{prop}
\begin{proof}
    We have $h(p)=h(q)$ if and only if $p \in P_v$ and $q \in P_w$ correspond to the same point in $P$ and $v-w \in \ker f$. The automorphism $\varphi$ that acts as $\varphi\colon P_u \to P_{u-v+w}$ via the identity map sends $p$ to $q$ and we get $h\circ \varphi = h$.
\end{proof}

\begin{example}
Let $P'$ be obtained from $P$ by merging some pairs of facets:
some examples are in Figure \ref{merge:fig}. Let $\lambda'$ be any colouring in $V'$ for $P'$. We construct a colouring $\lambda$ in $V = V'\times \matZ/2\matZ$ for $P$ simply by assigning the vectors $(\lambda(F),0)$ and $(\lambda(F),1)$ arbitrarily to the two facets of $P$ that are merged in the facet $F$ of $P'$. We get a degree 2 branched covering $N \to N'$, whose ramification locus is the codimension-two submanifold that is the preimage of the merging locus. 

The two examples in Figure \ref{merge:fig} are particularly interesting because in both cases $P$ and $P'$ can be given a right-angled hyperbolic and flat structure, respectively. Therefore we get a hyperbolic 2- or 3-manifold $N$ that is a double branched covering over a flat 2- or 3-manifold $N'$. The ramification locus consists of some points and some closed geodesics, respectively. The 3-dimensional example is a variation of a well-known construction of Thurston.
In both examples, by interpreting $P'$ as a hyperbolic right-angled $P$, we assign a hyperbolic $\pi$-orbifold structure to $N'$. 

Our main goal is now to construct some similar 4-dimensional examples.
\end{example}

\vspace{0.3cm}

\section{The construction of $M$}\label{mainexample}
We now apply the techniques introduced in the previous section to construct our 4-dimensional hyperbolic $\pi$-orbifold $M$. We start by investigating the geometry and combinatorics of the 120-cell.

\subsection{The 120-cell}
We recommend \cite{SchSeg} for a beautiful geometric introduction to the 120-cell.
Recall that the \emph{quaternion group} is
$$Q_8 = \{\pm 1, \pm i, \pm j, \pm k\}.$$
The \emph{binary tetrahedron group} is
$$T_{24}^* = Q_8 \sqcup \big\{ \tfrac 12(\pm 1 \pm i \pm j \pm k) \big\}.$$

Let $\phi = (1+\sqrt 5)/2$ be the golden ratio.  
The \emph{binary icosahedral group} $I_{120}^*$ is obtained by adding to $T_{24}^*$ the $8\times 12 = 96$ elements that are of the form 
$\frac 12 (a+bi+cj+dk)$ where $(a,b,c,d)$ is either
$(\pm \phi, \pm \phi^{-1}, \pm 1,0)$ or obtained from it via an even permutation. Note that $I_{120}^*$ consists of unit vectors; in fact, it is the preimage of the group of orientation-preserving isometries of the regular dodecahedron under the double cover $S^3\rightarrow SO(3)$, where $S^3$ is the set of unitary quaternions. The 600-cell is the convex hull of $I_{120}^*$. 

Let $K$ be the boundary of the 600-cell. Two vertices $q,q' \in I_{120}^*$ are adjacent (that is, they are connected by an edge) if and only if $\langle q,q' \rangle = \phi/2 \sim 0.81$. The 120-cell is the polyhedron $P$ obtained from the simplicial complex $K$. The facets of $P$ are thus $F_q$ with $q \in I_{120}^*$.

\subsubsection{The combinatorial Hopf fibration}
The element
$$p = \frac 12\big(\phi + \phi^{-1}i + j\big)$$
has order 10. Since $1$ and $p$ are adjacent in $K$, the 10 elements in the cyclic group $C=\langle p \rangle$ are the vertices of a decagon contained in the plane spanned by $1$ and $p$. The decagon corresponds to a ring of 10 dodecahedral facets of the 120-cell that form a solid torus $T_0$. Each dodecahedron intersects the two adjacent ones in two opposite faces.
The elements
\begin{gather*}
q_1 = \frac 12(\phi + i + \phi^{-1}k), \qquad q_2 = \frac 12(\phi + i - \phi^{-1}k), \qquad q_3 = \frac 12(\phi + \phi^{-1}j - k) \\
q_4 = \frac 12 ( \phi - \phi^{-1}i + j), \qquad q_5 = \frac 12(\phi + \phi^{-1}j + k)
\end{gather*}
are precisely the vertices in $K$ that are adjacent to both $1$ and $p$. The order-5 isometry $r(x) = pxp^{-1}$ fixes $1$ and $p$ and acts cyclically on these as
$$q_1 \to q_2 \to q_3 \to q_4 \to q_5 \to q_1.$$
The 6 distinct left cosets
$$C, \quad q_1C, \quad q_2C, \quad q_3C, \quad q_4C, \quad q_5C$$
are the vertices of 6 decagons, each contained in a different plane, which determine 6 disjoint rings $T_0, T_1, T_2, T_3, T_4, T_5$ of 10 dodecahedral facets each, in the 120-cell, each homeomorphic to a solid torus. Overall these 60 facets form a large solid torus, containing the smaller solid torus $T_0$ in its interior. See some beautiful pictures in \cite[Figure 6.2]{SchSeg} and \cite{rings}.

We can check that left multiplication by $k$ sends the plane containing $C$ to its orthogonal plane. Therefore it sends this large solid torus to the complement large solid torus. We obtain 5 more elements
\begin{gather*}
q_7 = kq_1 = \frac 12(-\phi^{-1} +j + \phi k), \quad q_8 = kq_2 = \frac 12(\phi^{-1} + j + \phi k), \\ 
q_9 = k q_3 = \frac 12(1 - \phi^{-1}i + \phi k), \quad
q_{10} = k q_4 = \frac 12(-i - \phi^{-1}j + \phi k), \\ 
q_{11} = k q_5 = \frac 12(-1 - \phi^{-1} i + \phi k).
\end{gather*}
We write $q_0 = 1$ and $q_6=k$. The 12 left cosets of $C$ are
$$q_0C, \quad \ldots \quad q_{11}C.$$
These determine the {rings} $T_{0}, \ldots, T_{11}$, {each homeomorphic to a solid torus} $S^1\times D^2$. The decomposition of the 120-cell into these 12 {rings} can be interpreted as a combinatorial Hopf fibration. 

There is an order 20 group $G = \matZ/4\matZ \times \matZ/5\matZ$ of isometries generated by left multiplication by $k$ and right multiplication by $p^2$, that preserves the decomposition into {rings} and acts freely and transitively on the 20 dodecahedra that lie in the two central {rings} $T_0 \cup T_6$.

The isometry $r(x) = pxp^{-1}$ generates another order 5 group $H = \langle r \rangle$ that fixes each dodecahedron in $T_0 \cup T_6$ and permutes cyclically the {rings} as
\begin{gather*}
T_1 \to T_2 \to T_3 \to T_4 \to T_5 \to T_1 \qquad
T_{11} \to T_{10} \to T_9 \to T_{8} \to T_{7} \to T_{11}.
\end{gather*}

\begin{remark} \label{nonadj:remark}
Two dodecahedra that lie in two non-successive {rings} in these cycles (e.g $T_2$ and $T_4$, or $T_7$ and $T_{10}$) are not adjacent.
\end{remark}

\begin{exercise} \label{vertex:exercise}
Every vertex of the 120-cell belongs to exactly 3 rings. (Hint: $3=150\times 12/600$)
\end{exercise}

The groups $G$ and $H$ generate altogether an order 100 group of isometries that preserve the combinatorial Hopf fibration.

\subsection{The hyperbolic 4-manifold $\tilde{M}$}
We now construct a closed hyperbolic 4-manifold $\tilde{M}$ by assigning a colouring $\lambda$ to the 120-cell. 

\subsubsection{The colouring}
Every element of $I_{120}^*$ and can be written uniquely as $q_ip^k$ for some $i\in \{0,\ldots, 11\}$ and $k\in\{0,\ldots, 9\}$. Set $V= (\matZ/2 \matZ)^7$ and
\begin{gather*}
\lambda (p^k) = \begin{cases} e_1+e_5+e_7 {\ \rm if \ } k {\ \rm odd}, \\ e_2+e_6+e_7 {\ \rm if \ } k {\ \rm even}, \end{cases} 
\ \lambda (q_1p^k) = \lambda (q_4p^k) = \begin{cases} e_1+e_5 {\ \rm if \ } k {\ \rm odd}, \\ e_2+e_6 {\ \rm if \ } k {\ \rm even}, \end{cases} \\
\lambda (q_2p^k) = \begin{cases} e_1+e_7 {\ \rm if \ } k {\ \rm odd}, \\ e_2+e_7 {\ \rm if \ } k {\ \rm even}, \end{cases} 
\  \lambda (q_3p^k) = \lambda (q_5p^k) = \begin{cases} e_1 {\ \rm if \ } k {\ \rm odd}, \\ e_2 {\ \rm if \ } k {\ \rm even}, \end{cases} \\
\lambda (q_6p^k) = \begin{cases} e_3+e_5+e_7 {\ \rm if \ } k {\ \rm odd}, \\ e_4+e_6+e_7 {\ \rm if \ } k {\ \rm even}, \end{cases} 
\!\!\!\! \lambda (q_7p^k) = \lambda (q_{10}p^k) = \begin{cases} e_3+e_5 {\ \rm if \ } k {\ \rm odd}, \\ e_4+e_6 {\ \rm if \ } k {\ \rm even}, \end{cases} \\
\lambda (q_8p^k) = \begin{cases} e_3+e_7 {\ \rm if \ } k {\ \rm odd}, \\ e_4+e_7 {\ \rm if \ } k {\ \rm even}, \end{cases} 
\  \lambda (q_9p^k) = \lambda (q_{11}p^k) = \begin{cases} e_3 {\ \rm if \ } k {\ \rm odd}, \\ e_4 {\ \rm if \ } k {\ \rm even}. \end{cases}\end{gather*}
We prove that this is indeed a colouring. First, by Remark \ref{nonadj:remark}, two adjacent vertices always get distinct colours. Then, Exercise \ref{vertex:exercise} easily implies that 4 vertices of a tetrahedron always get independent colours.

Interpreting the 120-cell $P$ as a right-angled hyperbolic polyhedron, the colouring $\lambda$ produces a hyperbolic manifold $\tilde{M}$ tessellated into $2^7$ right-angled hyperbolic 120-cells. {The orbifold Euler characteristic of $P$ is 
$$\chi(P) = \frac{600}{2^4} - \frac{1200}{2^3} + \frac{720}{2^2} - \frac{120}{2} + 1 = \frac{17}2$$ 
and therefore 
$$\chi(\tilde{M}) = 2^7 \times\frac{17}2 = 1088.$$
\begin{remark}
The colouring is symmetric with respect to the action of the group $G = \matZ/4\matZ \times \matZ/5\matZ$ {(that is, for every $g \in G$ there exists an isomorphism $f\colon V \to V$ such that $\lambda \circ g = f \circ \lambda$)} but not with respect to that of $H$.
\end{remark}

\begin{remark}{The rationale behind our choice of $\lambda$ is probably obscure at this point. The colouring $\lambda$ was obtained by refining some simpler colourings $\lambda'$ and $\lambda''$ that will be described in the next pages. The colouring $\lambda$ is a refinement of $\lambda'$, that is in turn a refinement of the very natural $\lambda''$.
}
\end{remark}

\subsection{The product $\Sigma \times \Sigma$ of two surfaces}
We now construct an eight-fold branched covering $\tilde{M} \to \Sigma \times \Sigma$ of our hyperbolic manifold $\tilde{M}$ onto a product of two surfaces. This is done by merging some facets of the 120-cell following the framework discussed in Section \ref{branchedcovering}.

\subsubsection{Merging the facets of the 120-cell}
We now consider the following partition of $I_{120}^*$ into 20 subsets of 6 elements each:
$$\big\{p^k, q_1p^k, q_2p^k, q_3p^k, q_4p^k, q_5p^k\big\}, 
\quad 
\big\{q_6p^k, q_7p^k, q_8p^k, q_9p^k, q_{10}p^k, q_{11}p^k\big\}.
$$

Here $k\in \{0, \ldots, 9\}$.
The partition is invariant under the actions of both groups $G$ and $H$.
Let $K''$ be the simplicial complex with 20 vertices obtained from $K$ by quotienting each set of the partition to a vertex, as explained in Section \ref{merging:subsubsection}. We can verify that the partition satisfies the collapsibility condition stated in Section \ref{merging:subsubsection}, and hence $K''$ gives rise to a polyhedron $P''$ obtained by merging some facets of $P$. (We use the symbols $K''$ and $P''$ instead of $K'$ and $P'$ because we will define some intermediate objects $K'$ and $P'$ below.) 

It turns out that $P''$ is a product of two decagons. The 20 new facets 
\begin{align*}
F_k'' & = F_{p^k} \cup F_{q_1p^k} \cup F_{q_2p^k} \cup F_{q_3p^k} \cup F_{q_4p^k} \cup F_{q_5p^k}, \\
F_{10+k}'' & = F_{q_6p^k} \cup F_{q_7p^k} \cup F_{q_8p^k} \cup F_{q_9p^k} \cup F_{q_{10}p^k} \cup F_{q_{11}p^k}
\end{align*}
with $k \in \{0,\ldots, 9\}$ are prisms with decagonal base, as shown in Figure \ref{dodecaedri:fig}, {and form two rings of 10 prisms each. We have $ P'' = D_1 \times D_2$ where $D_1,D_2$ are decagons and therefore
\begin{align*}
\partial P'' & = (D_1 \times \partial D_2) \cup (\partial D_1 \times D_2) \\
& = (F_0'' \cup \cdots \cup F_9'') \cup (F_{10}'' \cup \cdots \cup F_{19}'').\end{align*}}
The actions of $G$ and $H$ descend to $P''$. The order 20 group $G$ acts freely and transitively on the 20 facets of $P''$.

\begin{figure}
 \begin{center}
\centering
\labellist
\small\hair 2pt
\pinlabel $F_1$ at 135 233
\pinlabel $F_{q_1}$ at 140 100
\pinlabel $F_{q_2}$ at 257 150
\pinlabel $F_{q_3}$ at 203 280
\pinlabel $F_{q_4}$ at 70 280
\pinlabel $F_{q_5}$ at 16 150
\pinlabel $F_0''$ at 0 0
\endlabellist
  \includegraphics[width = 6 cm]{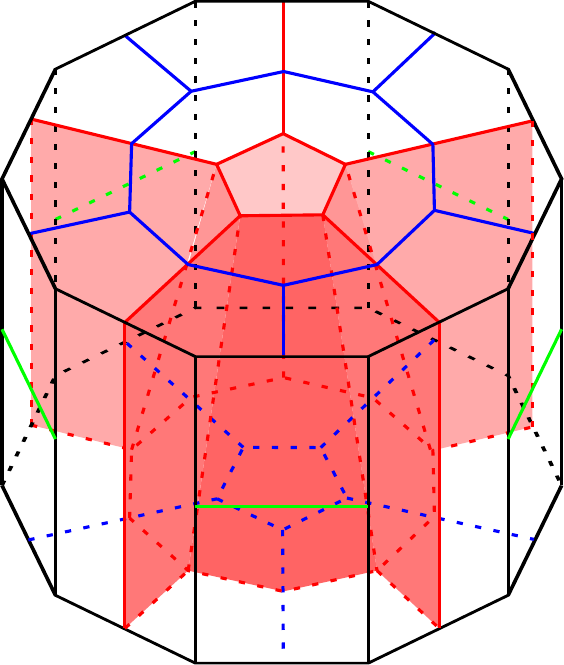}
 \end{center}
 \caption{The 6 dodecahedra $F_1, F_{q_1}, F_{q_2}, F_{q_3}, F_{q_4}, F_{q_5}$ merge to a prism $F_0''$ with decagonal base. Here, $1$ is the identity element in $I_{120}^*$. The 6 dodecahedra are separated by 10 red pentagons, that form the merging locus. The dodecahedron $F_1$ is central, the other five ones are lateral (the picture is symmetric under the obvious order 10 dihedral group action). The green and blue edges are contained in the merging loci of some adjacent prisms.}  \label{dodecaedri:fig}
\end{figure}

\subsubsection{The merging locus}
We determine the merging locus of $P''$.
We have $P'' = D_1 \times D_2$ where $D_1, D_2$ are the two identical decagons. Figure \ref{ramification:fig} shows two two-dimensional subsets $R, B \subset D_1 \times \partial D_2$, coloured in red and blue. 

Each $R$ and $B$ has 5 connected components, and each connected component is contained in a boundary prism $D_1 \times e$ for some edge $e \subset \partial D_2$. Consider $D_1 \times e$ as a $D_1$-bundle over $e$. Each connected component is a subbundle whose fiber is a circle with 5 prongs. The subbundle is of course trivial since the edge is contractible, but the circle shrinks as the point moves conterclockwise on the base edge as shown in Figure \ref{ramification:fig}.

Let $s$ be the symmetry $s(x,y) = (y,x)$ of $P'' = D_1 \times D_2$.

\begin{figure}
 \begin{center}
\centering
\labellist
\small\hair 2pt
\pinlabel $D_1$ at 20 770
\pinlabel $D_2$ at 265 770
\pinlabel $D_1$ at 515 770
\pinlabel $D_2$ at 760 770
\pinlabel $D_1$ at 20 515
\pinlabel $D_2$ at 265 515
\pinlabel $D_1$ at 515 515
\pinlabel $D_2$ at 760 515
\pinlabel $D_1$ at 20 255
\pinlabel $D_2$ at 265 255
\pinlabel $D_1$ at 515 255
\pinlabel $D_2$ at 760 255
\pinlabel $D_1$ at 20 0
\pinlabel $D_2$ at 265 0
\pinlabel $D_1$ at 515 0
\pinlabel $D_2$ at 760 0
\endlabellist
  \includegraphics[width = 12 cm]{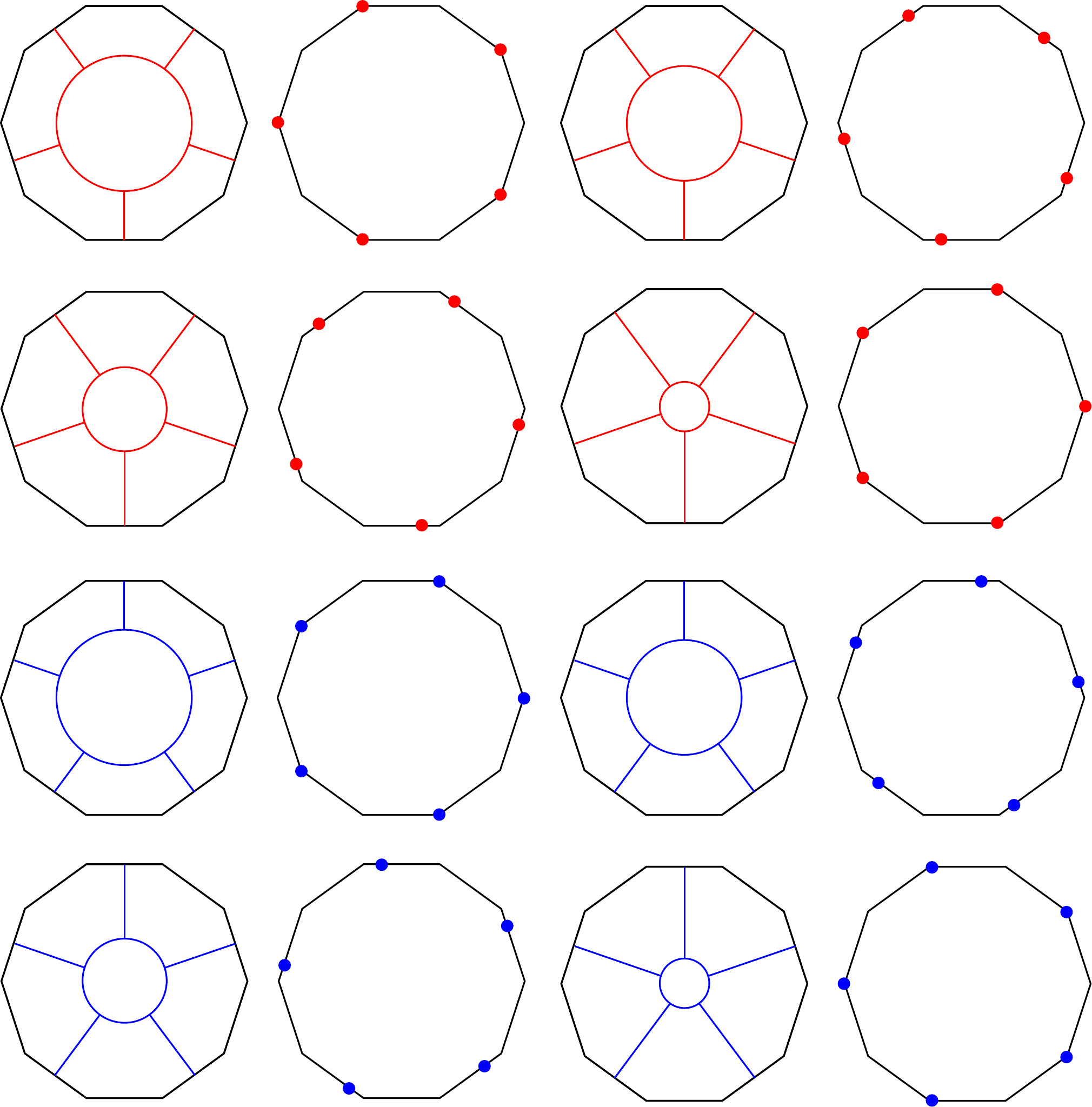}
 \end{center}
 \caption{The red and blue two-dimensional subsets $R, B \subset D_1 \cup \partial D_2$ have each 5 connected components. Each component is a subbundle of the product bundle $D_1 \times e$ over an edge $e$ of $D_2$, with fiber a circle with 5 prongs. The circle shrinks as the base point moves counterclockwise.}  \label{ramification:fig}
\end{figure}

\begin{prop}
The merging locus of $P'' = D_1 \times D_2$ is the subset $R \cup B \cup s(R) \cup s(B)$.       
\end{prop}
\begin{proof}
Figure \ref{dodecaedri:fig} shows that the merging locus on every prism $D_1 \times e$ projects onto $e$ as a fiber bundle with fiber a circle with five prongs, with the circle shrinking as the point in $e$ moves upwards.
\end{proof}

The merging locus intersects each dodecagonal face $D_1 \times \{v\}$, where $v$ is a vertex of $D_2$, into a graph as in Figure \ref{ramification_points:fig}. This shows that $R\cup B$ is connected, and $R\cap B$ consists of 50 points, 5 inside each dodecagonal face $D_1 \times \{v\}$. The merging locus has two connected components $R\cup B$ and $s(R) \cup s(B)$.

\begin{figure}
 \begin{center}
\centering
\labellist
\small\hair 2pt
\pinlabel $D_1$ at 20 0
\pinlabel $D_2$ at 265 0
\pinlabel $D_1$ at 515 0
\pinlabel $D_2$ at 760 0
\endlabellist
  \includegraphics[width = 12 cm]{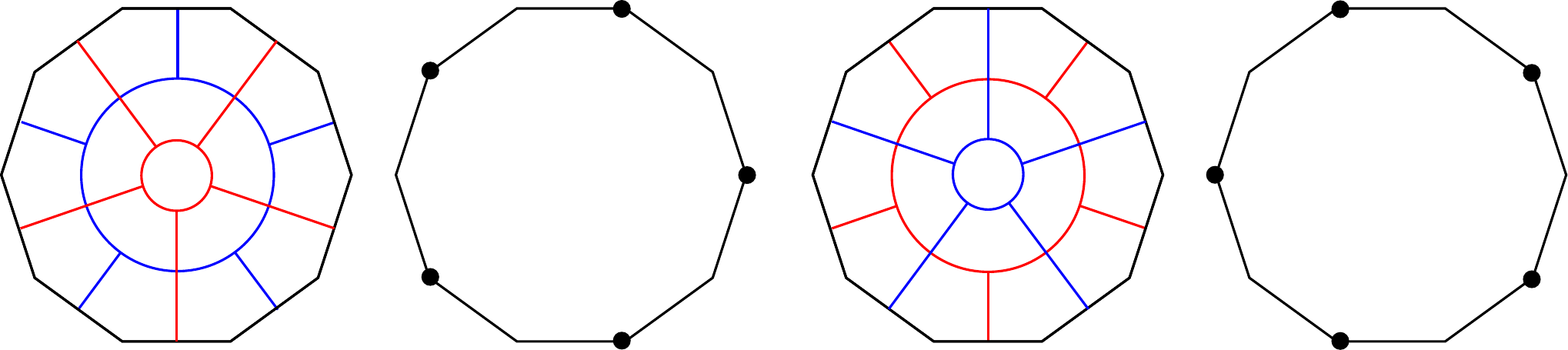}
 \end{center}
 \caption{The intersection of the merging locus with a dodecagonal face $D_1 \times \{v\}$ where $v$ is a vertex of $D_2$.}  \label{ramification_points:fig}
\end{figure}

\subsubsection{The colouring $\lambda''$}
We have discovered that $P''$ is a product of two decagons, and
we now colour its 20 facets in $V'' = (\matZ/2\matZ)^4$ as follows:
\begin{gather*}
\lambda'' (F_k'') = \begin{cases} e_1 {\ \rm if \ } k {\ \rm odd}, \\ e_2 {\ \rm if \ } k {\ \rm even}, \end{cases} \qquad
\lambda'' (F_{10+k}'') = \begin{cases} e_3 {\ \rm if \ } k {\ \rm odd}, \\ e_4 {\ \rm if \ } k {\ \rm even} \end{cases}
\end{gather*}
for $k\in \{0,\ldots, 9\}$.
This 4-colouring is the product of two standard 2-colourings of two decagons as in Figure \ref{decagons:fig}. Each such 2-colourings gives rise to a surface $\Sigma$ of genus 4. Therefore the colouring produces the product manifold $\Sigma\times \Sigma$, tessellated into $2^4=16$ copies of $P''$.

\begin{figure}
 \begin{center}
\centering
\labellist
\small\hair 2pt
\pinlabel $D_1$ at -10 0
\pinlabel $D_2$ at 470 0
\pinlabel $e_1$ at 40 5
\pinlabel $e_2$ at 110 -15
\pinlabel $e_1$ at 180 5
\pinlabel $e_2$ at 190 75
\pinlabel $e_1$ at 190 135
\pinlabel $e_2$ at 180 205
\pinlabel $e_1$ at 110 225
\pinlabel $e_2$ at 40 205
\pinlabel $e_1$ at -10 135
\pinlabel $e_2$ at -10 75
\pinlabel $e_3$ at 290 5
\pinlabel $e_4$ at 360 -15
\pinlabel $e_3$ at 430 5
\pinlabel $e_4$ at 480 75
\pinlabel $e_3$ at 480 135
\pinlabel $e_4$ at 430 205
\pinlabel $e_3$ at 360 225
\pinlabel $e_4$ at 290 205
\pinlabel $e_3$ at 280 135
\pinlabel $e_4$ at 280 75
\endlabellist
  \includegraphics[width = 6.5 cm]{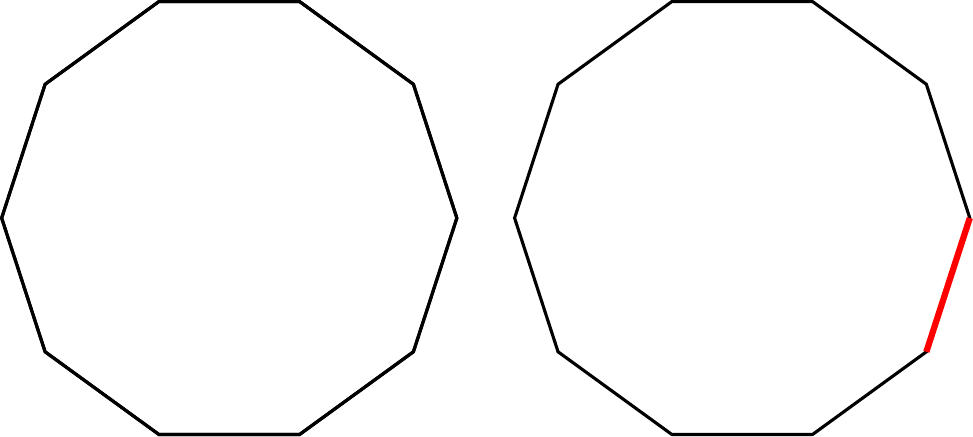}
  \vspace{.2 cm}
 \end{center}
 \caption{The 4-colouring $\lambda''$ on $P''$ is the product of two 2-colourings of two decagons}  \label{decagons:fig}
\end{figure}

The projection $f\colon (\matZ/2\matZ)^7 \to (\matZ/2\matZ)^4$ onto the first four coordinates is a homomorphism such that $\lambda''\circ g = f \circ \lambda$. From the discussion of Section \ref{branched:subsubsection}, we therefore obtain a branched covering $\tilde{M} \to \Sigma \times \Sigma$ of degree $2^{7-4} = 8$.

\subsubsection{The ramification locus in $\Sigma \times \Sigma$}
We determine the ramification locus of the branched cover $\tilde{M} \to \Sigma \times \Sigma$.\footnote{As a convention, throughout the paper `ramification locus' will refer to the subset of the codomain.} We are interested in particular in determining the types of ramification points. 

We will encounter in this paper three types of ramification points. Let $f\colon X \to Y$ be a piecewise-linear map between closed oriented 4-manifolds. We say that a point $x\in X$ is \emph{smooth}, \emph{simple}, \emph{triple}, and \emph{nodal}, if the map $f$ can be read on some oriented charts near $p$ respectively as
\begin{align*}
\matC^2 &\longrightarrow \matC^2, \qquad (z,w) \mapsto (z, w); \\
\matC^2 &\longrightarrow \matC^2, \qquad (z,w) \mapsto (z^2, w); \\
\matR^3 \times \matR &\longrightarrow \matR^3 / K \times \matR; \\
\matC^2 &\longrightarrow \matC^2, \qquad (z,w) \mapsto (z^2, w^2).
\end{align*}
In the third line, $K < {\rm SO}(3)$ is the Klein group $K \cong (\matZ/2\matZ)^2$ consisting of diagonal matrices with $\det =1$ and diagonal entries $\pm 1$; note that $\matR^3 / K$ is naturally identified with $\mathbb{R}^3$.
\par
The simple, triple, and nodal points are called \emph{singular}, and they form a two-dimensional space in $X$ which looks locally at a simple, triple, and nodal point like a disc, a product $Y \times I$, and two transverse discs, respectively.

{We say that a point $y \in Y$ is \emph{smooth, simple, triple}, or \emph{nodal}, if all its preimages are. All the branched coverings that we consider here arise from the construction described in Section \ref{branchedcovering} and are therefore regular, hence if one preimage of $y$ is smooth / simple / triple / nodal then all the other preimages of $y$ also are. 
}

\begin{prop}\label{isotropyprop}
    The ramification locus of the branched covering $\tilde{M} \to \Sigma \times \Sigma$ is a 2-dimensional subspace, {which consists of simple, triple, and nodal points only}. 
\end{prop}
\begin{proof}
    The ramification locus is the union of all the merging loci of all the products $D_1 \times D_2$ in the decomposition. The points in $R\cap B$ and $s(R) \cap s(B)$ give rise to nodes. The valence 3 vertices of the circles with 5 prongs form some closed curves entirely made of triple points. All the rest of the ramification locus consists of simple points.
\end{proof}

\subsection{Our main example $M$} We have constructed a branched covering $\tilde{M} \to \Sigma \times \Sigma$ of degree 8 with a quite complicated ramification locus. We now write this branched covering as a composition
$$\tilde{M} \stackrel {\tilde p } \longrightarrow M \stackrel p \longrightarrow \Sigma \times \Sigma$$
of two branched coverings of degree 2 and 4 respectively, both with simpler ramification locus. We start by introducing $M$.

The covering $\tilde{M} \to \Sigma \times \Sigma$ was derived from a merging map $P \to P''$ from the 120-cell $P$ to a product of two decagons $P''$. We now write this merging map as a composition of two merging maps
$$P \longrightarrow P' \longrightarrow P''$$
where $P'$ is some intermediate polyhedron that we now introduce.

\subsubsection{The intermediate polyhedron $P'$}
Recall that the vertices of $K$ form the binary icosahedral group $I_{120}^*$, and
every element in $I_{120}^*$ can be written uniquely as 
$q_ip^k$ for some $i\in \{0,\ldots, 11\}$ and $k\in\{0,\ldots, 9\}$. We now consider the following partition of $I_{120}^*$ into 60 subsets of order 3, 2, and 1: 
\begin{gather*}
\{ q_0p^k, q_1 p^k, q_4 p^k \}, \qquad \{q_2p^k, q_3 p^k\}, \qquad \{q_5p^k\}, \\
\{ q_6p^k, q_7 p^k, q_{10} p^k \}, \qquad \{q_8p^k, q_9 p^k\}, \qquad \{q_{11}p^k\}
\end{gather*}
as $k \in \{0, \ldots, 9\}$ varies.
This partition defines a new abstract simplicial complex $K'$ with $60$ vertices onto which $K$ surjects. We can verify that the collapsibility condition stated in Section \ref{merging:subsubsection} is verified, hence $|K'|$ is still a 3-sphere, and we get a new abstract polyhedron $P'$ with 60 facets, obtained by merging some facets of the 120-cell $P$. The facets of $P'$ are

\begin{align*}
F'_k  & = F_{q_0p^k} \cup F_{q_1 p^k} \cup F_{q_4 p^k}, \ & F_{10+k}' & = F_{q_2p^k} \cup F_{q_3 p^k}, \ & F_{20+k}' & = F_{q_5p^k}, \\
F_{30+k}' & =F_{q_6p^k} \cup F_{q_7 p^k} \cup F_{q_{10} p^k}, \ & F_{40+k}' & = F_{q_8p^k} \cup F_{q_9 p^k}, \ & F_{50+k}' & = F_{q_{11}p^k},
\end{align*}

for $k\in \{0,\ldots,9\}$.
The partition that defines $K''$ is actually a refinement of this one: hence we get surjective simplicial maps $K \to K' \to K''$ and also two merging homeomorphisms $P \to P' \to P''$.

\begin{figure}
 \begin{center}
\centering
\labellist
\small\hair 2pt
\pinlabel $F_{0}'$ at 140 100
\pinlabel $F_{10}'$ at 257 150
\pinlabel $F_{20}'$ at 14 150
\pinlabel $F_0''$ at 0 0
\endlabellist
  \includegraphics[width = 6 cm]{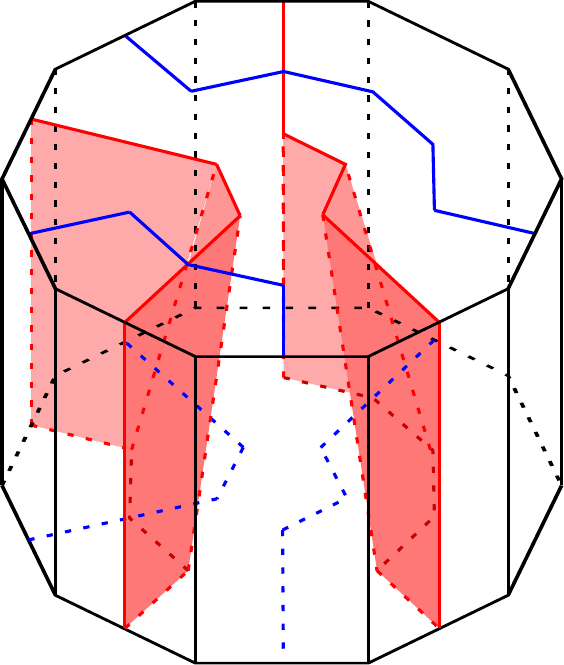}
 \end{center}
 \caption{The facets $F_0', F_{10}', F_{20}'$ are obtained by merging respectively 3, 2, and 1 dodecahedra of the 120-cell. The {three facets $F_0', F_{10}', F_{20}'$} then merge to form the prism $F_0''$ already shown in Figure \ref{dodecaedri:fig}. The faces separating the three facets $F_0', F_{10}', F_{20}'$ are drawn in red. The picture also shows, in blue, the edges of the analogous separating faces in the adjacent prisms $F_1''$ and $F_9''$ {that intersect $F_0''$ in the two base decagons}.
 }  \label{dodecaedri2:fig}
\end{figure}

{The polyhedron $P'$ has 300 vertices, 600 edges, 360 faces, and 60 facets. Therefore its orbifold Euler characteristic is
$$\chi(P') = \frac{300}{2^4} - \frac{600}{2^3} + \frac{360}{2^2} - \frac{60}2 + 1 = \frac{19}4.$$
}

\subsubsection{The colouring $\lambda'$}
We define a colouring $\lambda'$ of $P'$ in $(\matZ/2\matZ)^6$ as follows:

\begin{align*}
\lambda' (F_k') & = \begin{cases} e_1+e_5 {\ \rm if \ } k {\ \rm odd}, \\ e_2+e_6 {\ \rm if \ } k {\ \rm even}, \end{cases} \quad
& \lambda' (F_{10+k}') & = \lambda'(F_{20+k})= \begin{cases} e_1 {\ \rm if \ } k {\ \rm odd}, \\ e_2 {\ \rm if \ } k {\ \rm even}, \end{cases}  \\
\lambda' (F_{30+k}') & = \begin{cases} e_3+e_5 {\ \rm if \ } k {\ \rm odd}, \\ e_4+e_6 {\ \rm if \ } k {\ \rm even}, \end{cases} \quad
& \lambda' (F_{40+k}') & = \lambda'(F_{50+k}') = \begin{cases} e_3 {\ \rm if \ } k {\ \rm odd}, \\ e_4 {\ \rm if \ } k {\ \rm even} \end{cases}
\end{align*}
for $k\in \{0,\ldots,9\}$. This is indeed a colouring: adjacent vertices have distinct colours, and Exercise \ref{vertex:exercise} implies easily that any four different vectors are independent.
Let $M$ be the resulting 4-manifold. It is tessellated into $2^6$ copies of the polyhedron $P'$ {and therefore
$$\chi(M) = 2^6\cdot \chi(P') = 2^6 \times\frac{19}4 = 304.$$
The projections on the first factors
$$(\matZ/2\matZ)^7 \stackrel{f_1}\longrightarrow (\matZ/2\matZ)^6 \stackrel{f_2}\longrightarrow (\matZ/2\matZ)^4$$
are such that $f=f_2\circ f_1$ and moreover
$$\lambda'' \circ g_2 = f_2 \circ \lambda', \quad \lambda' \circ g_1 = f_1 \circ \lambda$$
where $g_2\colon K'' \to K', g_1 \colon K' \to K$ are the projections, and $g=g_1\circ g_2$.
As promised, we have decomposed our branched covering $\tilde{M} \to \Sigma \times \Sigma$ of order 8 into two branched coverings 

$$\tilde{M} \stackrel {\tilde p} \longrightarrow M \stackrel p \longrightarrow \Sigma \times \Sigma$$
of degree 2 and 4 respectively.

\subsection{The ramification locus}
We determine the ramification locus of both branched coverings $\tilde{M} \to M$ and $M \to \Sigma \times \Sigma$.
\subsubsection{The covering $M \to \Sigma \times \Sigma$}
A \emph{product torus} in $\Sigma \times \Sigma$ is a torus of type $\alpha \times \beta$ where $\alpha, \beta$ are simple closed curves in $\Sigma$.

\begin{prop}\label{80tori}
    The ramification locus of the four-fold branched cover $p\colon M \to \Sigma \times \Sigma$ {consists of simple and nodal points}. It is the union of 80 homologically non-trivial product tori that intersect transversely in 160 nodal points. Each torus intersect four other tori at a single point, and we can orient the tori so that the intersections are all positive for the product orientation on $\Sigma\times \Sigma$.
\end{prop}
{Crucially, there is no triple point.} As a sanity check, {we can recompute $\chi(M)$ using the branched covering, and we still get}
    $$\chi(M)=4\chi(\Sigma \times \Sigma) + 160 = 304.$$

\begin{proof}
    It follows from Figure \ref{dodecaedri2:fig} that the ramification locus intersects each prism into two squares, that are products $\alpha \times \beta$ and $\alpha' \times \beta$ of the arcs $\alpha,\alpha',\beta$ shown in Figure \ref{tori:fig}-(left). 
   {The two squares $\alpha\times \beta$ and $\alpha' \times \beta$ lift to $2+2=4$ tori $\tilde \alpha_i \times \beta$ and $\tilde \alpha_i' \times \beta$, in $\Sigma \times \Sigma$, that are shown in Figure \ref{surfaces:fig}.} Since $P''$ has 20 prism facets, we get $20 \times 2 \times 2 = 80$ tori overall. {The other tori are obtained from those shown in Figure \ref{surfaces:fig} by acting with rotations of order 10 and exchanging the two factors in $\Sigma \times \Sigma$.}

    Each torus intersects transversely 4 other tori, produced by squares in nearby prisms, as explained in Figure \ref{tori:fig}-(right). 
    {We orient the arcs as indicated in Figure \ref{tori:fig}-(right), and orient the curves correspondingly as shown in Figure \ref{surfaces:fig}. The tori inherit an orientation as products of oriented curves.}

    The transverse intersection between any pair of tori is {negative} 
    because the intersections of a red and a blue arc in the two decagons in Figure \ref{tori:fig}-(right) have {equal} signs ({equal signs yield negative intersections,} because if $(v_1,v_2)$ are tangent to the red arcs and $(w_1,w_2)$ are tangent to the blue arcs, we get that $(v_1,w_1,v_2,w_2)$ is positive, and hence $(v_1,v_2,w_1,w_2)$ is negative!) 
    {To transform the negative intersections into positive ones it suffices to modify the identification of the second factor in $\Sigma \times \Sigma$ with $\Sigma$ under an orientation-reversing map, for instance by mirroring the decagon $D_2$.}
\end{proof}

\begin{figure}
 \begin{center}
\centering
\labellist
\small\hair 2pt
\pinlabel $\alpha$ at 75 120
\pinlabel $\alpha'$ at 140 100
\pinlabel $\beta$ at 440 80
\endlabellist
  \includegraphics[width = 12 cm]{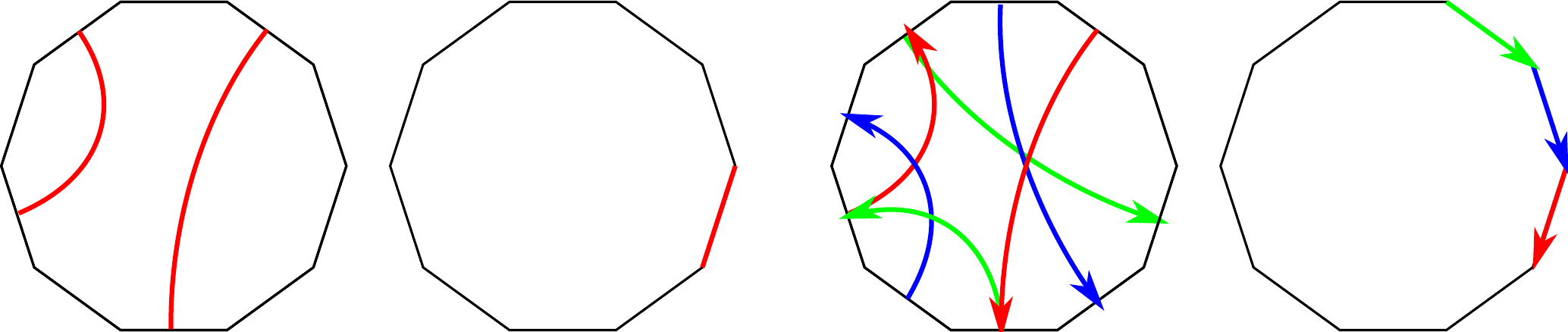}
 \end{center}
 \caption{The ramification locus in each prism consists of two squares $\alpha \times \beta$ and $\alpha' \times \beta$. 
 {Note that each arc $\alpha, \alpha'$ gives rise to two parallel curves $\tilde \alpha_i, \tilde \alpha'_i$ in $\Sigma$ via the colouring process, while $\beta$ gives rise to a single one $\tilde \beta$, see Figure \ref{surfaces:fig}}
 (left). Each blue torus intersects two red tori and two green tori, while the red and green tori are disjoint (right).
 }  \label{tori:fig}
\end{figure}

\begin{figure}
 \begin{center}
\centering
\labellist
\small\hair 2pt
\pinlabel $\tilde\alpha_1$ at 35 20
\pinlabel $\tilde\alpha_2$ at 42 70
\pinlabel $\tilde\alpha'_1$ at 152 70
\pinlabel $\tilde\alpha'_2$ at 140 20
\pinlabel $\tilde\beta$ at 260 60
\endlabellist
  \includegraphics[width = 12 cm]{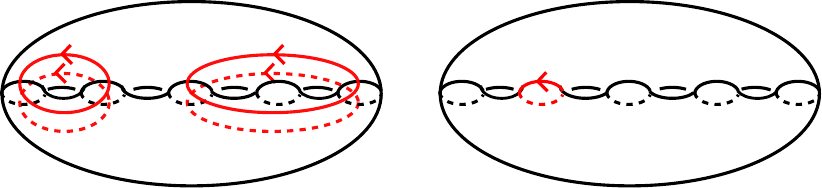}
 \end{center}
 \caption{{The preimages of the squares $\alpha \times \beta$ and $\alpha'\times \beta$ from Figure \ref{tori:fig}-(left) consist of $2+2=4$ product tori $\tilde \alpha_i \times \tilde\beta$ and $\tilde \alpha'_i \times \tilde\beta$ in $\Sigma \times \Sigma$. The curves are oriented as prescribed in Figure \ref{tori:fig}-(right), and the tori inherit an orientation.}
 }  \label{surfaces:fig}
\end{figure}

\subsubsection{A further factorization}
For later purposes, it will also be helpful to record the following.
\begin{prop}\label{twodoublebranch}
The four-fold branched covering of Proposition \ref{80tori} factors as the composition of two double branched coverings 
\begin{equation*}
M\stackrel{p_1}{\longrightarrow} N\stackrel{p_2}{\longrightarrow} \Sigma\times\Sigma
\end{equation*}
each only having simple singularities. In particular, the ramification loci of both $p_1$ and $p_2$ are a collection of disjoint embedded surfaces, and in the case of $p_2$ they are a subset the ramification locus of $p$ consisting of $40$ disjoint tori.
\end{prop}
\begin{proof}
From Figure \ref{tori:fig}-(left) we deduce that we can naturally bipartite the intersection graph of the 80 tori in Proposition \ref{80tori} so that intersecting tori are assigned different colors. This way, we obtain two collections of 40 tori, each consisting of 20 pairs of parallel tori, and $p_2$ is the double branched cover along one of these two collections.\end{proof}
{We can in fact build this double branched cover via one more polyhedron $P^*$ that is intermediate between $P'$ and $P'' = D_1\times D_2$. The polyhedron $P^*$ is obtained from $P'$ by merging the facets 
$$
F^*_k = F_k' \cup F_{10+k}' \cup F_{20+k}', \qquad
F^*_{10+k} = F_{30+k}' \cup F_{40+k}' \cup F_{50+k}'
$$
only for the even values $k\in\{0,2,4,6,8\}$. We then assign to $P^*$ the colouring $\lambda^*$ in $(\matZ/2\matZ)^5$ defined as follows:
\begin{align*}
\lambda^* (F_k') & = e_1+e_5 \quad
& \lambda^* (F_{10+k}') & = \lambda'(F_{20+k})=  e_1 \\
\lambda^* (F_{30+k}') & = e_3+e_5 \quad
& \lambda^* (F_{40+k}') & = \lambda'(F_{50+k}') =  e_3 \\
\lambda^* (F_{k-1}^*) & = e_2 \quad & \lambda^* (F_{k+9}^*) & = e_4 
\end{align*}
for every odd value $k\in \{1,3,5,7,9\}$.
}

\subsubsection{The covering $\tilde{M} \to M$}
Finally, we prove the following.

\begin{prop}
    The ramification locus of the double branched cover $\tilde{M} \to M$ consists of 160 orientable surfaces of genus 2 and 40 of genus 3, all with trivial normal bundle. 
    
    By interpreting $P'$ as a right-angled hyperbolic 120-cell $P''$ {with merged facets}, we assign to $M$ the structure of a hyperbolic $\pi$-orbifold {with singular set the ramification locus. The singular set} consists of totally geodesic surfaces. 
\end{prop}
{In particular, the singular set of this branched covering consists only of simple points, and again as a sanity check we compute 
    $$\chi(\tilde{M})=2\chi(M) - 160 (-2) - 40(-4) = 1088.$$
\begin{proof}
    The merging locus of $P'$ consists of the green faces $f_1,f_2,f_3$ shown in Figure \ref{dodecaedri3:fig}-(left), repeated in each prism. Figure \ref{dodecaedri3:fig}-(right) shows the colours of the edges of $f_1,f_2,f_3$ induced from $\lambda$. By what said in Section \ref{submanifolds:subsubsection} the preimage of $f_i$ in $\tilde{M}$ consists of some copies of the surface $S_i$ obtained from the pentagon $f_i$ equipped with the induced colouring in $V/V_i$, where $V = (\matZ/2\matZ)^7$ and $V_i$ is the subspace generated by the colourings of the two facets containing $f_i$. We have
    $$V_1 = V_2 = {\rm Span} (e_1+e_5, e_1+e_5+e_7), \quad V_3 = {\rm Span} (e_1, e_1+e_7)$$
    In $V/V_1 = V/V_2$ the pentagons $f_1$ and $f_2$ are coloured with 3 independent vectors, hence $S_1=S_2$ is tessellated into $2^3=8$ pentagons, hence $\chi(S_1)=\chi(S_2)=-2$ and $S_1,S_2$ are orientable surfaces of genus 2. In $V/V_3$ the pentagon $f_3$ is coloured with 4 independent vectors and hence $\chi(S_3)=-4$ and $S_3$ in an orientable surface of genus 3. 
    
    Every pentagon $f_1,f_2$ gives rise to $2^{5-3}=4$ surfaces of genus two, while each $f_3$ gives rise to $2^{5-4}=2$ surface of genus three. Since there are 20 prisms and hence 20 pentagons of each type, we get $20\times 2\times 4 = 160$ surfaces of genus two and $20 \times 2 = 40$ surfaces of genus three.
\end{proof}

\begin{figure}
 \begin{center}
\centering
\labellist
\small\hair 2pt
\pinlabel $F_{0}'$ at 140 100
\pinlabel $F_{10}'$ at 257 150
\pinlabel $F_{20}'$ at 14 150
\pinlabel $f_1$ at 127 200
\pinlabel $f_2$ at 122 234
\pinlabel $f_3$ at 210 200
\pinlabel $F_0''$ at 0 0
\pinlabel $f_1$ at 375 235
\pinlabel $f_2$ at 525 235
\pinlabel $f_3$ at 447 93
\pinlabel $e_1$ at 315 245
\pinlabel $e_1+e_7$ at 415 245
\pinlabel $e_2+e_6+e_7$ at 373 290
\pinlabel $e_2+e_6$ at 320 190
\pinlabel $e_2+e_7$ at 420 190
\pinlabel $e_1$ at 470 245
\pinlabel $e_1$ at 582 245
\pinlabel $e_2+e_6+e_7$ at 528 290
\pinlabel $e_2$ at 490 190
\pinlabel $e_2+e_6$ at 575 190
\pinlabel $e_2+e_6+e_7$ at 365 130
\pinlabel $e_1+e_5+e_7$ at 350 70
\pinlabel $e_2+e_7$ at 505 130
\pinlabel $e_2$ at 440 32
\pinlabel $v$ at 500 70
\endlabellist
  \includegraphics[width = 12 cm]{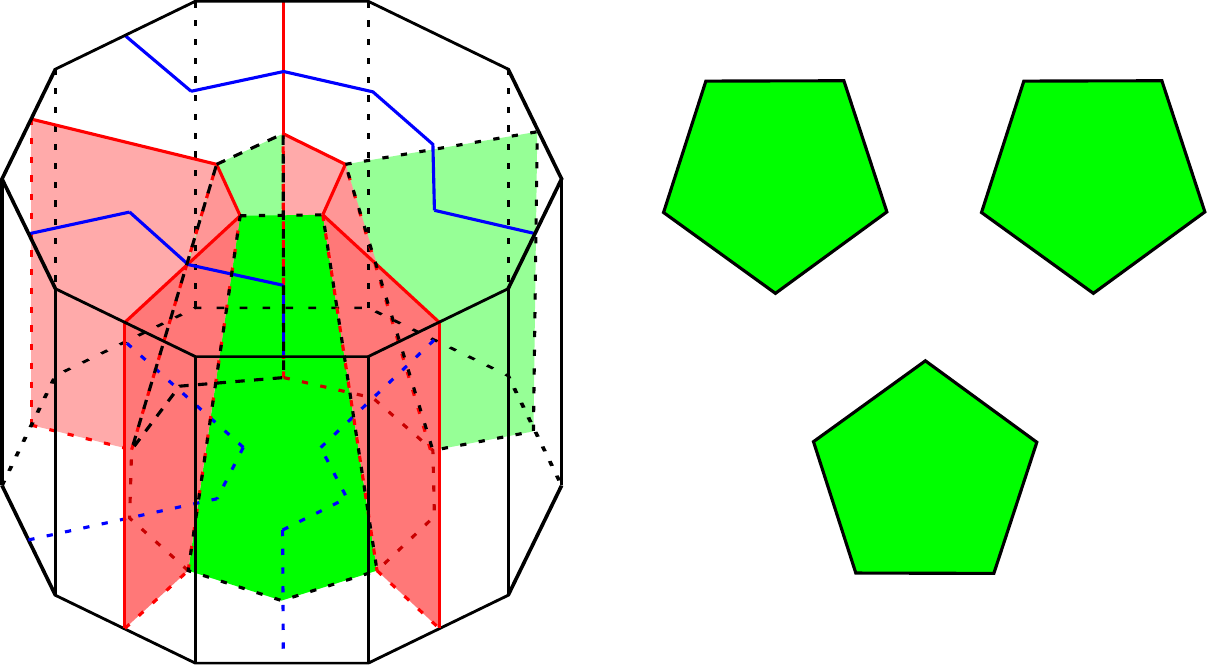}
 \end{center}
 \caption{The merging locus of $P'$ consists of the green faces $f_1,f_2,f_3$, repeated in each prism (left). The induced colouring of the green faces, where $v$ equals either $e_3+ \cdots$ or $e_4 + \cdots$ (right).}
 \label{dodecaedri3:fig}
\end{figure}

\vspace{0.3cm}

\section{The symplectic structure}\label{symplecticstuff}
In this final section we prove that the manifold $M$ we have described in Section \ref{mainexample} admits a symplectic structure. 
\par
The starting point of this is the topological description of $M$ we provided in Proposition \ref{80tori} as a four-fold branched cover $p\colon M\rightarrow \Sigma\times \Sigma$, with ramification locus a given collection of transversely intersecting tori. The key observation is that these are product tori, so in particular they are Lagrangian for a fixed product symplectic structure on $\Sigma\times \Sigma$. From this, we will show how to construct a symplectic form on $M$ via several steps which will follow by adapting constructions of (among the others) Auroux \cite{Aur} and Gompf \cite{Gom}; for the sake of clarity we will describe them in somewhat general terms before specializing to our case of interest.
\par
The fundamental principle underlying constructions in symplectic topology is the Moser technique, first introduced in \cite{Mos} to show that the total volume is the only invariant of a volume form up to isotopy. In symplectic geometry, this technique {is used} to deduce \textit{local} results starting from \textit{pointwise} assumptions (see for example \cite[Ch. 3]{MS}). For example, the proof of Darboux's theorem (symplectic manifolds $(X^{2n},\omega)$ are locally standard) starts from the linear algebra fact that all non-degenerate skew-symmetric bilinear forms on $\mathbb{R}^{2n}$ (with coordinates $x_1,y_1,\dots,x_n,y_n$) are equivalent to 
\begin{equation*}
\omega_{std}=dx_1\wedge dy_1+\dots +dx_n\wedge y_n.
\end{equation*}
Because of this, in the exposition that follows we will place a specific emphasis on ensuring pointwise conditions.

\subsection{Turning Lagrangian surfaces into symplectic ones} Recall that a surface $S$ (which we always assume to be oriented) in $(Y^{4},\omega_Y)$ is Lagrangian (resp. symplectic) if 
\begin{equation*}
\omega\lvert_{S}=0 \text{ (resp.}>0\text{) at every point.}
\end{equation*}
Of course, this implies the cohomological condition 
\begin{equation*}
\langle[\omega],[S]\rangle=0 \text{ (resp.}>0\text{),}
\end{equation*}
so that a Lagrangian surface cannot be isotopic to a symplectic one.
\par
Our construction begins with a variation {of} \cite[Lemma 1.6]{Gom}, which proves that (under a suitable cohomological condition) a collection of disjoint Lagrangians will be symplectic for a small perturbation of the original symplectic form. Our result generalizes this to transversely intersecting Lagrangians; while the result holds in greater generality, for simplicity of notation we will make assumptions fitting our case of interest (surfaces are tori such that each pair intersects in at most one point).
\begin{lemma}\label{coho}
Let $(Y,\omega_Y)$ be a closed symplectic $4$-manifold, and consider oriented embedded Lagrangian tori $\{T_i\}$ in $Y$ with the following properties:
\begin{enumerate}
\item the $T_i$ intersect each other transversely in at most one point;
\item there is a cohomology class $[\eta]\in H^2(Y;\mathbb{R})$ such that $\langle[\eta],[T_i]\rangle>0$ for all $i$. 
\end{enumerate}
  Then there is a symplectic form $\omega'_Y$ (arbitrarily close to $\omega_Y$) such that all the $T_i$ are symplectic.
\end{lemma}
Before proving this, let us recall some basic facts about Lagrangian submanifolds. First, by the Weinstein neighborhood theorem \cite[Theorem 3.4.13]{MS}, we can identify a tubular neighborhood $\nu L$ of a given Lagrangian $L$ in $(Y,\omega_Y)$ with a tubular neighborhood of $L$ in its cotangent bundle $T^*L$ equipped the its canonical symplectic form: for local coordinates $(y_1,y_2)$ on $L$, we call the coordinates $(x_1,x_2)$ for the cotangent fiber $T^*_pL$ in the ordered basis $\{dy_1,dy_2\}$, and the tautological on $T^*L$ form is given by the standard form $dx_1\wedge dy_1+dx_2\wedge dy_2$.
\par
In fact, given a collection of transversely intersecting Lagrangians $\{L_i\}$, by the same proof we can arrange the further property that at a given intersection point $p= L_i\cap L_j$ under these identifications the intersection $L_j\cap \nu L_i$ is a disk in the cotangent fiber $T^*_pL_i$. This is possible (via Moser's argument, cf. for example \cite{McRae}) because the symplectic group of a symplectic vector space acts transitively on pairs of transverse (unoriented) Lagrangians.
\par
Finally, because in our setup we are working with tori (so that the cotangent bundles are trivial), we can in fact simplify the notation by considering the local coordinates $(y_1,y_2)$ as induced by global ones (modulo some lattice in $\mathbb{R}^2$). Furthermore, we can choose our tubular neighborhoods to be products $\nu T_i\cong T_i\times \Delta_i$ (with $\Delta_i$ an open disk), and assume that the intersections between the Lagrangians occur along fibers $T_j\cap \nu T_i=\{p\}\times \Delta_i$.
\begin{proof}Consider a representative closed $2$-form $\eta$ for the cohomology class in assumption (2). Choose an area form $\omega_i\in \Omega^2(T_i)$ on each $T_i$ such that
\begin{equation*}
\int_{T_i}\eta=\int_{T_i}\omega_i>0,
\end{equation*}
so that one can find $\alpha_i\in \Omega^1(T_i)$ with
\begin{equation*}
d\alpha_i=\omega_i-\eta\lvert_{T_i}\in  \Omega^2(T_i).
\end{equation*}
We will also interpret $\alpha_i$ as a form on $\nu T_i$ via pull-back. Finally choose compactly supported functions
\begin{equation*}
f_i:\Delta_i\rightarrow \mathbb{R}
\end{equation*}
with $f_i(0)=1$, which we also interpret as a function on $\nu T_i$ by pull-back. We will consider the $1$-forms with compact support in $\nu T_i$ given by
\begin{equation*}
 \beta_i=f_i\alpha_i\in\Omega^1_c(\nu T_i)\subset \Omega^1(Y).    
\end{equation*}
In the local coordinates introduced above,  $\alpha_i$ lies in $\mathrm{span}\{dy_1,dy_2\}$, while $df_i$ lies in $\mathrm{span}\{dx_1,dx_2\}$ at every point.
Because of this, we see that when taking the exterior derivative
\begin{equation*}
d\beta_i=df_i\wedge \alpha_i+f_i\wedge d\alpha_i\in \Omega^2(Y)
\end{equation*}
the first summand restricts to zero on each $T_j$ and therefore
\begin{equation*}
d\beta_i\lvert_{T_j}=
\begin{cases}
\omega_i-\eta\lvert_{T_i}\text{ if }i=j\\
0 \text{ otherwise.}
\end{cases}
\end{equation*}
Because of this, the closed two-form
$$\eta'=\eta+\sum d\beta_i$$
restricts to the area form $\omega_i$ on each $T_i$, and for $\varepsilon>0$ sufficiently small the form 
\begin{equation*}
\omega'_Y=\omega_Y+\varepsilon\eta'\in\Omega^2(Y)
\end{equation*}
is a new symplectic form on $Y$ (being symplectic is an open condition) for which the $T_i$ are symplectic surfaces.\end{proof}

Note that in general the $T_i$ are not necessarily $\omega'_Y$-orthogonal; on the other hand, we can arrange this after a smooth isotopy under an additional positivity assumption (recall that symplectic manifolds are naturally oriented).
\begin{lemma}
In the situation of the previous lemma, assume further that all intersections between the Lagrangians are \textbf{positive}. Then there exists a symplectic form $\omega''_Y$ (smoothly isotopic to $\omega'_Y$) for which the original tori $\{T_i\}$ are all symplectic and orthogonal.
\end{lemma}
\begin{proof}
This follows from \cite[Lemma 2.3]{Gom} (which deals with a much more general situation) and its proof; in fact, we can arrange that the isotopy is supported in a neighborhood of each of the transverse intersections, with one of the two intersecting surfaces remaining fixed.
\par
This result is the differential version of the following linear algebra observation (which provides the pointwise model at the intersection point). Consider in $(\mathbb{C}^2,\omega_{std})$ a two-dimensional symplectic vector subspace $V_0$ which is positively transverse to $\mathbb{C}\times\{0\}$ (equipped with its natural complex orientation). By (positive) transversality, we can write (identifying each summand of $\mathbb{C}^2=\mathbb{C}\oplus\mathbb{C}$ with $\mathbb{R}^2$ with standard basis $\{e_1,e_2\}$) considering oriented spans throughout
\begin{equation*}
V_0=\mathrm{span}\{v_1\oplus e_1, v_2\oplus e_2\}\subset  \mathbb{C}\oplus\mathbb{C}
\end{equation*}
for some vectors $v_1,v_2\in\mathbb {{R}}^2$. The fact that $V_0$ is a positive symplectic subspace implies that (denoting by $\omega_0$ the standard form on $\mathbb{R}^2$, so that $\omega_{std}=\omega_0\oplus\omega_0$):
\begin{equation*}
0<\omega_{std}(v_1\oplus e_1, v_2\oplus e_2)=\omega_0(v_1,v_2)+1.  
\end{equation*}
From this, we see that because
\begin{equation*}
t^2\omega_0(v_1,v_2)+1>0\text{ for all }t\in[0,1],
\end{equation*}
the one parameter family of subspaces
\begin{equation*}
V_t=\mathrm{span}\{(tv_1)\oplus e_1, (tv_2)\oplus e_2\}
\end{equation*}
forms for $t\in[0,1]$ a path of symplectic subspaces positively transverse to $\mathbb{C}\times\{0\}$ connecting $V_1=V$ to $V_0=\{0\}\times\mathbb{C}$.
\par
As explained in detail in the proof of \cite[Lemma 2.3]{Gom}, from this one constructs a smooth isotopy $\varphi_t$ for $t\in[0,1]$ such that $\varphi_0=\mathrm{Id}$, for all $t\in[0,1]$ the $\{\varphi_t(T_i)\}$ form a collection of positively transverse $\omega_Y'$-symplectic tori, and the $\{\varphi_1(T_i)\}$ they are furthermore $\omega'_Y$-orthogonal. The claim of the Lemma then follows by applying the isotopy to the symplectic form instead of the Lagrangians.
\end{proof}

\begin{remark}
Note that in general two symplectic subspaces can be \textit{negatively} transverse. In the notation of the proof above can take for example
\begin{equation*}
\mathrm{span}(e_1\oplus 0, e_2\oplus 0)\text{ and } \mathrm{span}(v_1\oplus -e_1, v_2\oplus e_2)
\end{equation*}
with $\omega_0(v_1,v_2)>1$. On the other hand, orthogonal symplectic subspaces are necessarily positively intersecting, so that the condition in the lemma is necessary.
\end{remark}
The symplectic version of the Weinstein neighborhood theorem \cite[Theorem 3.4.10]{MS} implies the that each $T_i\subset (Y,\omega''_Y)$ has a neighborhood identified with $D(r_i)\times T_i$ (where $D(r_i)$ is an open disk  centered at $0$ in $\mathbb{R}^2$ of radius some $r_i>0$) equipped with the product symplectic structure $\omega_0\oplus \omega''_Y\lvert_{T_i}$.
Finally, by a final application of the Moser technique (being careful at the intersection points), we obtain the following local model result (cf. \cite[Proposition 3.5]{Sym} and \cite[Theorem 3.1]{GaySti}).
\begin{cor}\label{normal}
In the situation above, we can identify a neighborhood of
\begin{equation*}
\bigcup_i T_i\subset (Y,\omega''_Y)
\end{equation*}
with a corresponding plumbing of the product symplectic manifolds $(D(r_i)\times T_i, \omega_0\oplus \omega''_Y\lvert_{T_i})$ (for some choice of $\{r_i\}$).
\end{cor}
More in detail, by plumbing we mean that for each $i$ there is a collection of disjoint area-preserving embeddings $D(r_j)\hookrightarrow T_i$ (call the image $\Delta_{ij})$, one for each intersection point with an $T_j$, so that $D(r_i)\times \Delta_{ij}$ and $D(r_j)\times \Delta_{ji}$ are identified by swapping the two factors.

\vspace{0.3cm}

\subsection{Double branched covers}

In the setup above, consider now a (smooth) double branched cover
\begin{equation*}
f:X\rightarrow Y
\end{equation*}
which is branched along the subcollection $\{T^{br}_{\alpha}\}_{\alpha\in A}$ of the tori $\{T_i\}$; we will denote by $\{T^{nb}_\beta\}_{\beta\in B}$ the remaining ones.
\par
In the setup of Corollary \ref{normal}, we can assume after a smooth isotopy that for each $\alpha\in A$ the map $f$ in a suitable neighborhood $U_\alpha\equiv f^{-1}(T^{br}_{\alpha})\times \tilde{D_\alpha}$ of} $f^{-1}(T^{br}_{\alpha})\subset X$ 
is given by
\begin{align*}
\tilde{D}_\alpha\times f^{-1}(T^{br}_{\alpha}) &\rightarrow D_\alpha\times T^{br}_{\alpha}\\
(z,x)&\mapsto (z^2,f(x)),
\end{align*}
where the latter is the neighborhood in the plumbing given by Corollary \ref{normal}.
\\
\par
Of course, the pull-back form $f^*\omega_Y''$ is symplectic outside of the {preimage $f^{-1}(R)$ of the} ramification locus 
$R=\coprod T^{br}_{\alpha}$. 
Our goal is to prove the following.
\begin{prop}
In the setup above, there is a symplectic form $\omega_X$ which coincides with $f^*\omega_Y''$ outside a tubular neighborhood {of the preimage $f^{-1}(R)$} of the ramification locus $R\subset Y$ for which all the preimages $f^{-1}(T^{nb}_\beta)$ of the non-branching tori in the collection $\{T_i\}$ are symplectic submanifolds.
\end{prop}

The construction of the form $\omega_X$ is the one given in \cite{Aur} (whose analysis further includes cusp singularities in order to exhaust the singularity types of a generic holomorphic map). We reproduce it here in detail as we will additionally check that the preimages of the remaining submanifolds are symplectic.

\begin{proof}
Because of the work we have done regarding local models, and in particular Corollary \ref{normal}, we can focus on the simple case of the branched cover between neighborhoods of the origin in $\mathbb{C}^2$ (equipped with the standard symplectic form $\omega_{std}$) given by
\begin{equation*}
f(z_1,z_2)=(z_1^2,z_2).
\end{equation*}
In this model the {preimage of the} ramification locus is given by 
\begin{equation*}
\{z_1=0\}=\{0\}\times \mathbb{C}\quad(=R),
\end{equation*}
while its symplectic orthogonal (which is the model for the preimage remaining non-branching $T_i$, near the intersection points) is given by
\begin{equation}\label{wantsymp}
V:=\{z_2=0\}=\mathbb{C}\times \{0\}=\mathrm{ker}df_{(0,0)}\quad (=f^{-1}(T^{nr})).
\end{equation}
One readily computes that in these coordinates ($z_j=x_j+iy_j$, $j=1,2$)
\begin{equation*}
f^*\omega_{std}=4(x_1^2+y_1^2)\cdot dx_1\wedge dy_1+dx_2\wedge dy_2,
\end{equation*}
so that it vanishes on $V$ at the origin. In order to remedy this we will adjust it by exact forms defined in the local models  so that we get a symplectic form on $\mathbb{C}^2$ for which $V$ is a symplectic submanifold.
\par
Given $p\in R$, looking at the local models consider the $2$-form
\begin{equation*}
\alpha_p=d\left[\chi_1(|z_1|)\chi_2(|z_2|) x_1dy_1\right]
\end{equation*}
where
\begin{equation*}
\chi_i:\mathbb{R}^{\geq0}\rightarrow \mathbb{R}^{\geq0}
\end{equation*}
are smooth non-increasing compactly supported functions which are $1$ near $0$, the specific shape of which is to be decided. Explicitly, $\alpha_p$ involves three terms of which only two only involve $dx_1\wedge dy_1$, namely
\begin{equation*}
\chi_2(|z_2|)\chi_1(|z_1|)\cdot dx_1\wedge dy_1\text{ and }x_1\chi_2(|z_2|)\cdot d\chi_1(|z_1|)\wedge dy_1
\end{equation*}
Note that the first form coincides with $dx_1\wedge dy_1$ in a neighborhood of $p$, while the latter vanishes identically in a neighborhood of $p$. Therefore by compactness, summing the forms corresponding to a suitable finite set of $R$ (and for suitable choices of the $\chi_i$), we obtain a closed $2$-form $\alpha$ on $X$ with the property that $f^*\omega_Y+\varepsilon \alpha$ is readily checked to be a symplectic form on $X$ for which the preimages of the non-branching tori is symplectic.\end{proof}

\subsection{Conclusion of the proof of the main Theorem.}
By Proposition \ref{80tori}, our main example $M$ admits a four-fold branched cover
\begin{equation*}
p\colon  M\rightarrow \Sigma\times\Sigma    
\end{equation*}
which by Proposition \ref{twodoublebranch} factors as a composition of two double branched covers
\begin{equation*}
M\stackrel{p_2}{\longrightarrow} N \stackrel{p_1}{\longrightarrow}\Sigma\times\Sigma 
\end{equation*}
The image of the ramification locus of $p$ in $\Sigma\times \Sigma$ is a collection of 80 oriented product tori which are therefore Lagrangian for any product symplectic structure. Furthermore, they intersect each other positively in at most one point. We need to check that the \textit{global} assumption (2) in Lemma \ref{coho} holds in our situation.
\begin{lemma}
For the collection of Lagrangian tori $\{T_i\}\subset\Sigma\times\Sigma $ forming image of the ramification locus of $p$, there exists a cohomology class $[\eta]$ such that $\langle[\eta],[T_i]\rangle>0$ for all $i$.
\end{lemma}
\begin{proof}
Such a class $[\eta]$ exists if and only there exists an open half space containing all classes $\{[T_i]\}\subset H_2(\Sigma\times \Sigma;\mathbb{R})$, or, equivalently, there is no non-trivial linear combination
\begin{equation*}
\lambda_{i_1}[T_{i_1}]+\dots+\lambda_{i_k}[T_{i_k}]=0
\end{equation*}
with all coefficients positive. Indeed, such a linear combination cannot exists, because by Proposition \ref{80tori} there is a class $[T_j]$ such that $[T_j]\cdot [T_{i_1}]=1$ (and $[T_j]\cdot [T_{i_r}]\geq 0$ for $r=1\dots k$), so that taking intersections one obtains $\lambda_{i_1}=0$.
\end{proof}

Therefore we can apply the constructions in the previous two subsections to the double branched cover $p_1$ to obtain a symplectic structure on $N$ for which the preimage of the non-branching tori of $p_1$ (which are the image of the ramification locus of $p_2$) are symplectic. We can then apply directly the construction of Auroux discussed in the previous subsection to the double branched cover $p_2$ to obtain the desired symplectic structure on $M$.
\\
\par
Finally, to obtain infinitely many examples in our main Theorem, one can simply pull-back the branched cover $p$ along covering spaces $f:B\rightarrow \Sigma\times\Sigma$.

\vspace{0.5cm}

\bibliographystyle{alpha}
\bibliography{biblio}

\end{document}